\documentclass[a4paper, 10pt, reqno]{amsart}
\usepackage[usenames, dvipsnames]{color}
\usepackage{graphics, graphicx}
\usepackage{amsmath,amssymb,latexsym, amsfonts}%
\usepackage{amsthm}%
\newtheorem{theorem}{Theorem}[section]
\newtheorem{lemma}[theorem]{Lemma}

\allowdisplaybreaks
\newcommand{\ol}[1]{#1^{-1}}
\newcommand{\Xpm}{X^\pm}
\newcommand{\CF}{\mathcal{CF}}
\DeclareMathOperator{\FIM}{FIM}
\DeclareMathOperator{\FG}{FG}
\begin{document}

\title{Free inverse monoids are co-context-free}
\author[T.M. Brough]{Tara Macalister Brough}
\address{Centro de Matemática e Aplicações, Universidade Nova de Lisboa}
\email{tarabrough@gmail.com}
\author[M. Johnson]{Marianne Johnson}
\address{Department of Mathematics, University of Manchester}
\email{marianne.johnson@manchester.ac.uk}
\author[M. Kambites]{Mark Kambites}
\address{Department of Mathematics, University of Manchester}
\email{mark.kambites@manchester.ac.uk}
\author[C-F. Nyberg-Brodda]{Carl-Fredrik Nyberg-Brodda}
\address{June E. Huh Centre for Mathematical Challenges, Korea Institute for Advanced Study (KIAS)}
\email{cfnb@kias.re.kr}

\begin{abstract}
We prove (using grammars) that the free inverse monoid of every finite rank has co-context-free word problem.  Equivalently, the co-word problem of the free inverse monoid of every finite rank is context-free.
\end{abstract}

\date{\today}
\maketitle

\section{Introduction}
The word problem of a monoid $M$ with respect to a finite generating set
$A$ is the language  ${\rm WP}(M, A) = \{u\#v^{\rm rev} : u =_S v, u, v \in A^*\},$
where $\#$ is a symbol not in $A$ and $v^{\rm rev}$ denotes the reverse of the word $v$. If $M$ is an inverse monoid and $A$ is a symmetric generating set for $M$, then (as far as language-theoretic properties are concerned) the above problem is equivalent to
$\{u\#\ol{v} : u =_S v, u, v \in A^*\},$
where $u\#\ol{v}$ is obtained from $u\#v^{\rm rev}$ by replacing every symbol after the $\#$ by its inverse.   Word problem languages for free inverse monoids of finite rank were studied by the first author \cite{Brough2018}, who showed that these languages are not context-free, nor (except for rank~1) even an intersection of finitely many context-free languages, but they are recognised by checking stack automata (and hence ET0L). In this short note, we prove that the word problem of a free inverse monoid is always \emph{co$\CF$} (the complement of a context-free language).  This was shown for rank $1$ in \cite{Brough2018}, but not considered there for higher ranks.

For groups, the word problem being co$\CF$ is equivalent to the co-word problem (the complement of the word problem) being context-free, and this property is independent of the choice of generating set, hence it makes sense to refer to a group with co$\CF$ word problem as a \emph{co$\CF$ group}.  These groups were introduced in \cite{Holt2005}, and include several interesting examples such as the Higman--Thompson groups \cite{Lehnert2007}.

For an inverse monoid $M$, we define the co-word problem of $M$ with respect to $A$ as ${\rm coWP}(M,A) = \{u\#v^{-1}: u\neq_M v, u,v\in (A^\pm)^*\}$.
While this is not complementary to ${\rm WP}(M,A)$ in $(A^\pm)^*$, the words which are in neither the word problem nor the co-word problem form a regular language, and thus the word problem being co$\CF$ is still equivalent to the co-word problem being context-free.
Moreover, these properties are still independent of the choice of generating set, and so we may speak of a 
\emph{co$\CF$ monoid} in the same way as for groups.



\section{Preliminaries}
Free inverse monoids are for our purposes here most naturally defined by a description first hinted at by Scheiblich \cite{Scheiblich1973} and given a geometric interpretation by Munn \cite{Munn1974} in terms of what are now called \textit{Munn trees}.
Let $X$ be a set and for each $x \in X$ let $\ol{x}$ be a new symbol not contained in $X$, let $\ol{X} = \{\ol{x}: x \in X\}$ and $\Xpm = X \cup \ol{X}$. 
The \emph{Munn tree} of a word $w \in (\Xpm)^*$ is a bi-rooted tree $(\varepsilon, T_w, g_w)$, where $T_w$ is the subtree of the Cayley graph of the free group $\FG(X)$ obtained as the set of vertices and edges when starting at vertex $\varepsilon$ and reading $w = w_1\cdots w_n$ from left to right following the forward edge labelled $w_i$ if $w_i \in X$ and the backwards edge labelled $w_i$ if $w_i \in \ol{X}$, and $g_w$ is the  final vertex of the path traced. The vertices of each Munn tree are therefore labelled by elements of the free group (i.e. reduced words in the generators $\Xpm$), and it is clear from the description above that the Munn tree $T_w$ contains vertices $\varepsilon$ and $g_w$, as the start and terminal vertex of the path traced by reading the word through this graph.  The relations of the free inverse monoid $\FIM(X)$ ensure that two words produce the same Munn tree if and only if they represent the same element of $\FIM(X)$. Thus if $m \in \FIM(X)$ we may unambiguously write $T_m$ for the corresponding Munn tree. Moreover, the product of two elements in $\FIM(X)$ can easily be computed using their corresponding Munn trees: the product of $(\varepsilon, T_w, g_w)$ and $(\varepsilon, T_v, g_v)$ is $(\varepsilon, T_w \cup g_wT_v, g_wg_v)$, where $g_wT_v$ denotes the image of $T_v$ under the obvious left action of
$g_w$ and the union is taken inside the Cayley graph of $\FG(X)$. 

\section{Idempotents avoiding a rooted branch} 

The idempotents in $\FIM(X)$ are all elements whose Munn trees have terminal vertex $\varepsilon$, and so in particular any word representing an idempotent has even length, and moreover any two idempotent elements commute. For any $x\in \Xpm$, we say that an idempotent $e$ in $\FIM(X)$ \emph{avoids $x$} if $T_e$ does not contain the edge between the vertices $\varepsilon$ and $x$.
We write $L({\mathcal E})$ and $L({\mathcal E}_x)$ for
the sublanguages of $(\Xpm)^*$ representing, respectively, all idempotents in $\FIM(X)$, and the idempotents in $\FIM(X)$ avoiding $x$.

\begin{lemma}\label{idpt}
Let $X$ be a finite set and let $x \in \Xpm$.
\begin{enumerate}
	\item The language $L({\mathcal E})$ is generated 
	by the following grammar, with start symbol $E$:
	\[ E\rightarrow EE \mid xE\ol{x} \mid \varepsilon, \quad x\in \Xpm\]
\item The language $L({\mathcal E}_x)$ is generated 
by the grammar given by the following productions for all $a \in \Xpm$ and with start symbol $Z_x$:
\[ Z_a\rightarrow Z_a Z_a \mid y Z_{\ol{y}} \ol{y} \mid \varepsilon, \quad y\in \Xpm\setminus \{a\}.\]
\end{enumerate}  
\end{lemma}
\begin{proof}
(1) Clearly each  word generated by the grammar is an idempotent element of $\FIM(X)$. We prove the converse by induction on the (even) length of words representing an idempotent of $\FIM(X)$. The empty word is produced by the grammar. Suppose every idempotent element of $\FIM(X)$ that can be expressed as a product of strictly fewer than $2m$ generators is generated by the grammar, and let $u$ be an idempotent which can be expressed as a product of $2m$ but no fewer generators, say $u=u_1 \cdots u_{2m}$ where each $u_i \in \Xpm$. If a proper prefix $u_1\cdots u_i$ with $i<2m$ is idempotent, then by considering the corresponding Munn trees it is clear that $u_{i+1}\cdots u_{2m}$ must also be idempotent (to ensure that the terminal vertex is $\varepsilon$). If no proper prefix of $u$ is idempotent, then the path traced by $u$ in the Cayley graph of the free group returns to the start vertex $\varepsilon$ exactly once, and so it follows that $u_{2m} = u_1^{-1}$, and $u_2 \cdots u_{2m-1}$ is idempotent. 
In both cases, we have by induction that $u$ is generated by the above grammar.

(2) Similarly, each word generated by the given grammar with start symbol $Z_x$ is an idempotent avoiding $x$, whilst if $u$ is an idempotent avoiding $x$ then either (i) $u$ is the empty word, (ii) $u$ factorises as a product of two idempotents each avoiding $x$, or else (iii) the path traced by $u$ in the Cayley graph of the free group returns to the start vertex $\varepsilon$ exactly once, so that $u=yey^{-1}$  for some $y \in \Xpm$ and some idempotent $e$. In the first two cases it is clear that $u$ is produced by the grammar. In the third case, since $u$ avoids $x$ we must have $y\neq x$, and since the path of $u$ returns to the start vertex exactly once we must also have that $e$ avoids $y^{-1}$. The result now follows by induction on the length of $u$. 
\end{proof}

\section{Main Result}
\begin{theorem}
The free inverse monoid of every finite rank has co-context-free word problem.
\end{theorem}
\begin{proof}
Let $X$ be a finite set. Recalling that for words $u,v \in \Xpm$ we have ${u=_{\FIM(X)} v}$ if and only if the Munn trees $(\varepsilon, T_u, g_u)$ and $(\varepsilon, T_v, g_v)$ are equal, we can express the co-word problem of $\FIM(X)$ as a union of languages
\begin{eqnarray*}
{\rm coWP}(\FIM(X), \Xpm) &=&  K_1 \cup K_2 \cup \, {\rm coWP}(\FG(X), \Xpm) 
\end{eqnarray*} where $K_1$ and $K_2$ are the following subsets:
$$K_1 = \{u\#\ol{v} : u =_{{\rm FG}(X)} v \textrm{ and } T_u\mbox{ contains an edge that } T_v \mbox { does not }\}$$
$$K_2 = \{u\#\ol{v} : u =_{{\rm FG}(X)} v \textrm{ and } T_v\mbox{ contains an edge that } T_u \mbox { does not} \}.$$
Free groups are examples of \emph{co$\CF$ groups} \cite{Holt2005}, meaning that the group-theoretic co-word problem is context-free, from which it follows that ${\rm coWP}(\FG(X), \Xpm)$ is also context-free. Since a union of context-free languages is context-free,  noting the symmetry between $K_1$ and $K_2$ it therefore suffices to produce a context-free grammar for $K_1$.

For $m,n\geq 0$, let $K_{m,n}$ denote the set of all words of the form $u \# \ol{v}$, where \begin{eqnarray}
	u &=& e_1 x_1 \cdots e_m x_m \,\, p_0 x  p_1 \ol{x} p_2 \,\,y_1 f_1 \cdots y_n f_n,\label{u}\\
	v &=& e_1' x_1 \cdots e_m' x_m \;\;\qquad q_x \qquad \,y_1 f'_1 \cdots y_n f'_n,\label{v}
\end{eqnarray}
for some $x_1, \ldots, x_m, x, y_1,\ldots, y_n\in \Xpm$ such that $w:=
x_1\cdots x_m$ and $y:=y_1\cdots y_n$ are reduced words with $x_m \neq \ol{x}$ and $y_1 \neq x$, and the remaining symbols on the right hand side all denote words in $L(\mathcal{E})$, with $e_i'\in L({\mathcal E}_{x_i})$ for $1 \leq i \leq m$, $f'_i\in L({\mathcal E}_{\ol{y}_i})$ for $1 \leq i \leq n$ and $q_x\in L({\mathcal E}_x)$. Here the prefixes $e_1\cdots x_m$ and $e_1'\cdots x_m$ are interpreted as the empty word if $m=0$, and likewise for the suffixes $y_1\cdots f_n$ and $y_1\cdots f_n'$ if $n=0$. We show $K_1 = \bigcup_{m,n} K_{m,n}$. 

It is clear that $K_{0,0} \subseteq K_1$. Suppose that $K_{a,b} \subseteq K_1$ for all $a, b\geq 0$ with $a+b <d$, and let $u\#\ol{v} \in K_{m,n}$ for some $m, n \geq 0$ with $m+n=d$. Thus $u$ and $v$ satisfy equations \eqref{u} and \eqref{v}, subject to the conditions on each factor, and hence in particular $u=_{\FG(x)} v$. By assumption,  $w=x_1\cdots x_m$ and $y=y_1\cdots y_n$ are reduced words and since $x_m \neq x^{-1}$ and $y_1 \neq x$, we note that $wx$ and $x^{-1}y$ are also reduced words. The conditions on the idempotents $e_i'$  and $f_i'$ ensure that the given decomposition of $v$ highlights the \emph{first} occurrence of each $x_i$ 
on the path from $\varepsilon$ to $w$ in $T_v$, and the \emph{last} occurrence of each $y_i$ on the path from $w$ to the terminal point of $T_v$. This, combined with the fact that $q_x$ avoids $x$ and $y_1 \neq x$, ensures that $T_v$ does not contain the edge $(w,wx)$.  Meanwhile, the path corresponding to the subword $x p_1 \ol{x}$ of $u$ starts at vertex $w$, ensuring that $T_u$ contains the edge $(w,wx)$.  Thus, $u \#\ol{v} \in K_1$. 

Conversely, for $u\# \ol{v}\in K_1$, let $w$ and $wx$ be reduced words such that $x \in \Xpm$ and $(w, wx)$ is closest to the root $\varepsilon$ among edges contained in $T_u$ but not in $T_v$.  That is, $w$ (and hence each edge 
on the path labelled by $w$) is contained in both $T_u$ and $T_v$, but the edge labelled by $x$ and starting at $w$ is in $T_u$ but not in $T_v$.  
Let $w = x_1\ldots x_m$ with $x_i\in \Xpm$. We have $x_{i+1}\neq \ol{x}_i$ for all $i$ since $w$ is reduced, and also $x_m\neq \ol{x}$ since $wx$ is reduced. Since $u =_{{\rm FG}(X)} v$ and the reduced word $w$ is a vertex of both $T_u$ and $T_v$ we have $u=e_1x_1\cdots e_m x_m U$ and $v=e_1'x_1\cdots e_m' x_m V$, where $e_i, e_i'$ are idempotents and $ U=_{\FG(X)} V$.  Moreover, the idempotents $e_i'$ may be chosen so that $e_i' \in L(\mathcal{E}_{x_i})$, by insisting that each $x_i$ identified in our decomposition of $v$ is the first possible. By assumption, $T_U$ contains the edge $(\varepsilon, x)$, giving that  $p_0x$ is a prefix of $U$ for some $p_0 \in L(\mathcal{E})$. Moreover, since $U =_{\FG(X)} V$ and $T_V$ does not contain the edge $(\varepsilon,x)$, the path of $U$ must return to vertex $\varepsilon$ after reaching vertex $x$, giving $U=p_0xp_1x^{-1}p_2A$, where $p_0, p_1, p_2 \in L(\mathcal{E})$, with $p_2$ chosen maximally so that no prefix of $A$ is idempotent. Thus $A=y_1f_1\cdots y_nf_n$ for some reduced word $y_1\cdots y_n$ and idempotents $f_1, \ldots, f_n$. Since $U =_{\FG(X)} V=_{\FG(X)}y_1 \cdots y_n$, we must have $V=q_x y_1f_1'\cdots y_nf_n'$ for some $q_x, f_i \in L(\mathcal{E})$ and we may choose $f_i' \in L(\mathcal{E}_{\ol{y_i}})$, by insisting that each $y_i$ identified in our decomposition of $V$ is the last possible. Finally, since $T_V$ does not contain the edge $(\varepsilon,x)$ it is clear that $q_x \in L(\mathcal{E}_x)$ and $y_1 \neq x$. This completes the proof of the claim that $K_1 = \bigcup_{m,n} K_{m,n}$.

We claim now that the context-free grammar $\Gamma$ with set of terminals ${\Xpm \cup \{\#\}}$, non-terminals $\{S, E, P_x,Q_x,Z_x: x \in \Xpm\}$, start symbol $S$, and given by the following productions (for all $x,y\in \Xpm$, subject to the given restrictions) generates $K_1$:
\begin{eqnarray}
	\label{S}S \rightarrow P_x,&& \\ 
	\label{P} P_x \rightarrow E x P_y \ol{x} Z_x, && y \neq \ol{x}\\
	\label{P2}P_x \rightarrow E x E \ol{x} E Q_y Z_x, && y\neq x\\
	\label{Q} Q_x \rightarrow x E Q_y Z_{\ol{x}} \ol{x} \mid \#,  &&  y \neq \ol{x}\\  
	\label{E} E \rightarrow EE \mid xE\ol{x} \mid \varepsilon &&\\ 
	\label{Z} Z_x \rightarrow Z_{x} Z_{x} \mid y Z_{\ol{y}} \ol{y} \mid \varepsilon, &&  y \neq x.
\end{eqnarray}
It is straightforward to see that each word generated by $\Gamma$ can be obtained by first applying \eqref{S}, followed by $m\geq 0$ applications of \eqref{P}, one application of~\eqref{P2}, $n \geq 0$ applications of the first alternative of \eqref{Q}, one application of the second alternative of \eqref{Q}, and concluding with some number applications of \eqref{E} and \eqref{Z}; let us call such any such derivation of a word an $(m,n)$-derivation. For each fixed $m,n \geq 0$ we demonstrate that the language $L_{m,n}$ of all words produced by $(m,n)$-derivations in $\Gamma$ is equal to $K_{m,n}$, and hence $\Gamma$ is a context-free grammar for $K_1$. Let $m,n \geq 1$. By definition, every $(m,n)$-derivation begins as follows:
\begin{eqnarray*}
S \hspace{-2mm} &\rightarrow& \hspace{-2mm} P_{x_1} \rightarrow Ex_1P_{x_2}\ol{x_1}Z_{x_1} \rightarrow \cdots \rightarrow Ex_1\cdots Ex_{m}P_{x}\ol{x_{m}}Z_{x_{m}}\cdots \ol{x_1}Z_{x_1}\\
& \rightarrow& \hspace{-2mm} Ex_1\cdots Ex_{m}ExE\ol{x}EQ_{y_1}Z_{x}\ol{x_{m}}Z_{x_{m}}\cdots \ol{x_1}Z_{x_1}\\
& \rightarrow& \hspace{-2mm} Ex_1\cdots Ex_{m}ExE\ol{x}Ey_1EQ_{y_2}Z_{\ol{y_1}}\ol{y_1}Z_{x}\ol{x_{m}}Z_{x_{m}}\cdots \ol{x_1}Z_{x_1}\\
&\cdots&\\
& \rightarrow& \hspace{-2mm} Ex_1\cdots Ex_{m}ExE\ol{x}Ey_1\cdots Ey_nE\#Z_{\ol{y_n}}\ol{y_n} \cdots Z_{\ol{y_1}}\ol{y_1}Z_{x}\ol{x_{m}}Z_{x_{m}}\cdots \ol{x_1}Z_{x_1},
\end{eqnarray*}
where $x_1, \ldots, x_m, x, y_1,\ldots, y_n \in \Xpm$ are such that $x_{i+1} \neq \ol{x_{i}}$ for $i=1, \ldots, m-1$, $x_m \neq \ol{x}$, $y_1 \neq x$ and $y_{i+1} \neq y_i$ for $i=1, \ldots n-1$. 
Note that the productions in \eqref{E} and \eqref{Z} are the same as those given in Lemma~\ref{idpt}, and hence the non-terminals $E$ and $Z_x$ produce 
the languages $L({\mathcal E})$ and $L({\mathcal E}_x)$ respectively. It is now easily seen that the words generated by $(m,n)$-derivations are exactly those of the form in the definition of $K_{m,n}$, so that $L_{m,n}=K_{m,n}$ for all $m, n \geq 1$. If $m=0$ or $n=0$, then an entirely similar argument (omitting application of the corresponding productions) demonstrates that $L_{m,n}=K_{m,n}$. 
\end{proof}

\section*{Acknowledgements}
{\small The first author was supported by a Scheme 4 grant from the London Mathematical Society during a visit to Manchester, where this work was carried out. The fourth author was at the time of the visit employed by the University of Manchester and supported by the Dame Kathleen Ollerenshaw Trust, and is currently supported by the Mid-Career Researcher Program (RS-2023-00278510) through the National Research Foundation funded by the government of Korea, and by the KIAS Individual Grant MG094701 at Korea Institute for Advanced Study.}

\end{document}